\documentclass[reqno, 11pt, oneside]{amsart}

\usepackage{latexsym, amsmath,amssymb}
\usepackage{hyperref}
\hypersetup{colorlinks=true, citecolor=cyan, linkcolor=blue}

\usepackage[initials, alphabetic,msc-links]{amsrefs}
\usepackage{graphicx}
\usepackage{xcolor}
\usepackage{enumitem}

\usepackage[latin1]{inputenc}

 \numberwithin{equation}{section}

\usepackage{amsmath}

\def\Xint#1{\mathchoice
{\XXint\displaystyle\textstyle{#1}}%
{\XXint\textstyle\scriptstyle{#1}}%
{\XXint\scriptstyle\scriptscriptstyle{#1}}%
{\XXint\scriptscriptstyle% 
\scriptscriptstyle{#1}}%
\!\int}
\def\XXint#1#2#3{{\setbox0=\hbox{$#1{#2#3}{%
\int}$ }
\vcenter{\hbox{$#2#3$ }}\kern-.6\wd0}}

\def\dashint{\Xint-}

\renewcommand{\epsilon}{\varepsilon}
\newtheorem{theorem}{Theorem}

\newtheorem{lemma}[theorem]{Lemma}

\newtheorem{proposition}[theorem]{Proposition}

\newtheorem{remark}[theorem]{Remark}
\newcommand{\bth}{\begin{theorem}}
\newcommand{\ble}{\begin{lemma}}
\newcommand{\bcor}{\begin{corr}}

\newcommand{\bdeff}{\begin{deff}}

\newcommand{\bprop}{\begin{proposition}}
\newcommand{\ele}{\end{lemma}}
\newcommand{\ecor}{\end{corr}}
\newcommand{\edeff}{\end{deff}}

\numberwithin{theorem}{section}

\newcommand{\eprop}{\end{proposition}}

\newcommand{\supp}{\text{supp }}
\renewcommand{\Pi}{\varPi}

\renewcommand{\epsilon}{\varepsilon}

\newcommand{\R}{{\mathbb R}}

\newcommand{\pa}[1]{\left(#1\right)}
\newcommand{\norm}[1]{\left\|#1\right\|}

\usepackage{esint}

\usepackage{comment}

\allowdisplaybreaks

\def\vint_#1{\mathchoice%
        {\mathop{\kern 0.2em\vrule width 0.6em height 0.69678ex depth -0.58065ex
                \kern -0.8em \intop}\nolimits_{\kern -0.4em#1}}%
        {\mathop{\kern 0.1em\vrule width 0.5em height 0.69678ex depth -0.60387ex
                \kern -0.6em \intop}\nolimits_{#1}}%
        {\mathop{\kern 0.1em\vrule width 0.5em height 0.69678ex
            depth -0.60387ex
                \kern -0.6em \intop}\nolimits_{#1}}%
        {\mathop{\kern 0.1em\vrule width 0.5em height 0.69678ex depth -0.60387ex
                \kern -0.6em \intop}\nolimits_{#1}}}
\def\vintslides_#1{\mathchoice%
        {\mathop{\kern 0.1em\vrule width 0.5em height 0.697ex depth -0.581ex
                \kern -0.6em \intop}\nolimits_{\kern -0.4em#1}}%
        {\mathop{\kern 0.1em\vrule width 0.3em height 0.697ex depth -0.604ex
                \kern -0.4em \intop}\nolimits_{#1}}%
        {\mathop{\kern 0.1em\vrule width 0.3em height 0.697ex depth -0.604ex
                \kern -0.4em \intop}\nolimits_{#1}}%
        {\mathop{\kern 0.1em\vrule width 0.3em height 0.697ex depth -0.604ex
                \kern -0.4em \intop}\nolimits_{#1}}}

\newcommand{\aveint}[2]{\mathchoice%
        {\mathop{\kern 0.2em\vrule width 0.6em height 0.69678ex depth -0.58065ex
                \kern -0.8em \intop}\nolimits_{\kern -0.45em#1}^{#2}}%
        {\mathop{\kern 0.1em\vrule width 0.5em height 0.69678ex depth -0.60387ex
                \kern -0.6em \intop}\nolimits_{#1}^{#2}}%
        {\mathop{\kern 0.1em\vrule width 0.5em height 0.69678ex depth -0.60387ex
                \kern -0.6em \intop}\nolimits_{#1}^{#2}}%
        {\mathop{\kern 0.1em\vrule width 0.5em height 0.69678ex depth -0.60387ex
                \kern -0.6em \intop}\nolimits_{#1}^{#2}}}

\newcommand{\dist}{\operatorname{dist}}

\def\Xint#1{\mathchoice
  {\XXint\displaystyle\textstyle{#1}}%
  {\XXint\textstyle\scriptstyle{#1}}%
  {\XXint\scriptstyle\scriptscriptstyle{#1}}%
  {\XXint\scriptscriptstyle\scriptscriptstyle{#1}}%
  \!\int}
\def\XXint#1#2#3{{\setbox0=\hbox{$#1{#2#3}{\int}$}
    \vcenter{\hbox{$#2#3$}}\kern-.5\wd0}}

\newcommand{\vertiii}[1]{{\left\vert\kern-0.25ex\left\vert\kern-0.25ex\left\vert #1 
    \right\vert\kern-0.25ex\right\vert\kern-0.25ex\right\vert}}

\newcommand{\vertii}[1]{{\left\vert\kern-0.25ex\left\vert\kern-0.25ex  #1 
    \kern-0.25ex\right\vert\kern-0.25ex\right\vert}}

\makeatletter
\@namedef{subjclassname@2020}{\textup{2020} Mathematics Subject Classification}
\makeatother

\usepackage{mathabx}
\usepackage{geometry}

\begin{document}

\title[Bourgain--Brezis--Mironescu formula for Riesz Potentials]
{Bourgain--Brezis--Mironescu formula for Riesz Potentials}

\author[A. Claros]{Alejandro Claros}

\address[A. Claros]{BCAM -- Basque Center for Applied Mathematics, Bilbao, Spain\newline
Universidad del Pa\'is Vasco / Euskal Herriko Unibertsitatea (UPV/EHU), Bilbao, Spain}

\email{aclaros@bcamath.org, aclaros003@ikasle.ehu.eus}

\author[C. P\'erez]{Carlos P\'erez}

\address[C. P\'erez]{BCAM -- Basque Center for Applied Mathematics, Bilbao, Spain\newline
Universidad del Pa\'is Vasco / Euskal Herriko Unibertsitatea (UPV/EHU), Bilbao, Spain\newline
Ikerbasque, Bilbao, Spain}

\email{cperez@bcamath.org}

\thanks{A. Claros is supported by the Basque Government through the BERC 2022-2025 program, by the Ministry of Science and Innovation through Grant PRE2021-099091 funded by BCAM Severo Ochoa accreditation CEX2021-001142-S/MICIN/AEI/10.13039/501100011033 and by ESF+, and by the project PID2023-146646NB-I00 funded by MICIU/AEI/10.13039/501100011033 and by ESF+. 
\\ \hspace*{1.5em}C. P\'erez is supported by the Spanish government through the grant PID2023-146646NB-I00 and by Severo Ochoa accreditation CEX2021-001142-S, both at BCAM, and also by the Basque Government through grant IT1615-22 at the University of the Basque Country and by the BERC programme 2022-2025 at BCAM}


%\date{\today}

%\dedicatory{}


\begin{abstract}
We identify the Bourgain--Brezis--Mironescu pointwise limit of the nonlocal potential operator $(1-\alpha)\, I_\alpha(\mathcal D^\alpha f)$, $0<\alpha<1$, where $I_\alpha$ denotes the Riesz potential and $\mathcal D^\alpha$ a nonlinear fractional differential operator. Specifically, for every $f\in C_c^\infty(\R^n)$ and every $x\in\R^n$, we show that
\begin{equation*}
	\lim_{\alpha\to 1^-} (1-\alpha)\, I_\alpha(\mathcal D^\alpha f)(x) = K_n\, I_1(|\nabla f|)(x),
\end{equation*}
where $K_n$ is the geometric constant appearing in the well-known Bourgain--Brezis--Mironescu formula \cite{BBM2}. By a density argument, we further extend this result to every $f\in W^{1,1}(\R^n)$, obtaining almost everywhere convergence along subsequences.
\end{abstract}

\maketitle

\section{Introduction and statement of the main result}

A classical starting point in Sobolev theory is the subrepresentation formula
\begin{equation}\label{eq:subrep-classical}
|f(x)| \lesssim_n I_1(|\nabla f|)(x),
\end{equation}
valid for $f\in C_c^\infty(\R^n)$, with $n\ge2$.  Here, for $\alpha\in (0,n)$, the Riesz potential is defined by
\begin{equation*}
I_\alpha f(x):=\gamma_{n,\alpha}\int_{\R^n}\frac{f(y)}{|x-y|^{n-\alpha}}\,dy,
\end{equation*}
where $\gamma_{n,\alpha}$ is a normalization constant chosen so that $\mathcal F(I_\alpha f)(\xi)=|\xi|^{-\alpha}\mathcal F(f)(\xi)$, where $\mathcal F(f)(\xi)=\int_{\R^n} f(x) \, e^{-ix\cdot \xi}\,dx$ denotes the Fourier transform.  Explicitly, this constant is given by 
\begin{equation}\label{Defgamma}
	\gamma_{n,\alpha}=\frac{\Gamma\bigl(\tfrac{n-\alpha}{2}\bigr)}{2^\alpha \pi^{\frac{n}{2}}\Gamma\bigl(\tfrac{\alpha}{2}\bigr)}.
\end{equation}
The estimate \eqref{eq:subrep-classical} can be derived from the classical local $(1,1)$-Poincar\'e inequality on balls (or cubes) via a standard chain argument, which converts local mean oscillation control into a pointwise potential estimate. Recently, several extensions of \eqref{eq:subrep-classical} have been obtained. On the one hand, the authors in \cite{HMP2} refine the left-hand side through the substitution $f \mapsto Tf$, with $T$ ranging from various maximal operators to singular integrals with rough kernels. On the other hand, \cite{HMP} improves the right-hand side by replacing it with a smaller operator (see \eqref{eq:HMP} below). This improvement of \eqref{eq:subrep-classical} is  closely related to the very interesting nonlinear fractional differential operator introduced by D. Spector in \cite{SpectorJFA}:
\begin{equation*}
\mathcal{D}^\alpha f(x):=\int_{\R^n}\frac{|f(x)-f(y)|}{|x-y|^{n+\alpha}},dy,
\qquad 0<\alpha<1,
\end{equation*}
which  serves as a nonlocal ``fractional gradient" of order $\alpha$. This operator is of interest for several reasons, one of them is the recent fractional subrepresentation formula established in \cite{HMP} (see \eqref{eq:frac-subrep} below). On the other hand, it preserves some of the structural properties of first-order differential operators. For instance, it satisfies the following Leibniz-type inequality (see \cite{NahasPonce}),
\begin{equation*}
	\norm{\mathcal{D}^\alpha (fg)}_{L^1(\R^n)} \le \norm{f\, \mathcal{D}^\alpha g}_{L^1(\R^n)} +\norm{g\, \mathcal{D}^\alpha f}_{L^1(\R^n)}.
\end{equation*}

The link between the fractional scale and the gradient scale is given by the Bourgain--Brezis--Mironescu (BBM) formula \cite{BBM2}\footnote{The original proof of the BBM formula was stated and proved in \cite{BBM2} for smooth bounded domains $\Omega \subset \R^n$. The result can be extended to the whole $\R^n$, see \cite[Appendix A]{BSY} and \cite{Kaushik}.}. For $f\in C^\infty_c(\R^n)$, it identifies the limit of the fractional Gagliardo seminorm
\begin{equation}\label{eq:BBM-limit}
\lim_{\alpha\to1^-}(1-\alpha)\int_{\R^n}\int_{\R^n}\frac{|f(x)-f(y)|}{|x-y|^{n+\alpha}}\,dx\,dy =K_n\int_{\R^n}|\nabla f(x)|\,dx,
\end{equation}
where
\begin{equation}\label{eq:Kn}
K_n:=\int_{\mathbb{S}^{n-1}}|\omega\cdot e|\,d\sigma(\omega)
\end{equation}
and $e\in \mathbb{S}^{n-1}$ is arbitrary. In particular, the factor $(1-\alpha)$ is crucial, otherwise the limit is not even finite for nonconstant functions (see \cite{BrezisConstant}). BBM-type limits have also been investigated in extension domains \cite{KKP}, in arbitrary bounded domains, where boundary effects require modifications of the inner integral in the seminorm \cite{DrelichmanDuran}, and in fully arbitrary domains with $W^{\alpha,p}_q$ seminorms \cite{Kaushik}, which are related to Triebel--Lizorkin spaces. Further extensions have been obtained in fractional Orlicz--Sobolev spaces \cite{Orlicz,Orlicz2,Orlicz3}, in the setting of interpolation spaces \cite{DominguezMilman}, in ball Banach function spaces \cite{DachunYang, ZhuYangYuan2023}, in connection with Triebel--Lizorkin spaces \cite{BSY}, and for anisotropic fractional energies \cite{FernandezSalort}.

In \cite{BBM1}, the following fractional Poincar\'e inequality on cubes is proved with the BBM extra term $(1-\alpha)$. For every cube $Q\subset\R^n$ and every $f\in C^\infty(Q)$,
\begin{equation}\label{eq:BBM-Poincare}
\dashint_Q |f(x)-f_Q|\,dx \lesssim  (1-\alpha)\,\ell(Q)^\alpha\, \dashint_Q\int_Q \frac{|f(x)-f(y)|}{|x-y|^{n+\alpha}}\,dx\,dy, \qquad 0<\alpha<1,
\end{equation}
where $f_Q=\dashint_Q f$. Once \eqref{eq:BBM-Poincare} is available, a standard telescoping argument (see for instance \cite{FLW, HMPV}) yields the fractional subrepresentation formula
\begin{equation}\label{eq:frac-subrep}
|f(x)| \lesssim_n (1-\alpha)\, I_\alpha(\mathcal{D}^\alpha f)(x),
\qquad x\in\R^n,\quad f\in C^\infty_c(\R^n).
\end{equation}
as shown in \cite{HMP}.  In particular, $(1-\alpha)I_\alpha(\mathcal{D}^\alpha f)$ plays the role of a fractional version of the first-order Riesz potential of the gradient, in the spirit of \eqref{eq:subrep-classical}.  Interesting related subrepresentation formulae were established for $s$-John domains in \cite{DD2}, although without the corresponding BBM-type factor  $(1-\alpha)$.

An alternative motivation for studying $(1-\alpha)I_\alpha(\mathcal{D}^\alpha f)$ stems from the theory of rough singular integrals. In the recent work \cite{HMP}, it is established that, for every $\alpha \in (0,1)$, if $\Omega$ is homogeneous of degree zero, has vanishing average on $\mathbb{S}^{n-1}$, and belongs to the Marcinkiewicz space $L^{\frac{n}{\alpha},\infty}(\mathbb{S}^{n-1})$, then the maximal truncated rough operator
\begin{equation*}
	T^\star_{\Omega} f(x):=\sup_{\varepsilon>0}\Bigl|\int_{|x-y|>\varepsilon}\frac{\Omega(x-y)}{|x-y|^n}\,f(y)\,dy\Bigr|
\end{equation*}
satisfies the pointwise estimate
\begin{equation}\label{HMPsubrep}
T^\star_{\Omega} f(x)\leq c_n\, \|\Omega\|_{L^{\frac{n}{\alpha},\infty}(\mathbb S^{n-1})}\, (1-\alpha) \,I_{\alpha}(\mathcal{D}^{\alpha}f)(x), \qquad f\in C^\infty_c(\R^n).
\end{equation}
The same work also establishes that for every $f\in C_c^\infty(\R^n)$,
\begin{equation}\label{eq:HMP}
(1-\alpha)\, I_\alpha(\mathcal{D}^\alpha f)(x)\le C_{n,\alpha}\, I_1(|\nabla f|)(x),
\end{equation}
where the constant $C_{n,\alpha}$ bounded as $\alpha\to1^-$. 
Combining this with \eqref{eq:frac-subrep}, we obtain
\begin{equation*}
	|f(x)| \lesssim_n (1-\alpha)\, I_\alpha(\mathcal{D}^\alpha f)(x) \le C_{n,\alpha}\, I_1(|\nabla f|)(x)
\end{equation*}
for every $x\in \R^n$. The boundedness of $C_{n,\alpha}$ when $\alpha\to1^-$ in \eqref{eq:HMP} suggests that the operator $(1-\alpha)I_\alpha(\mathcal{D}^\alpha f)$ should have a meaningful limit as $\alpha\to 1^-$, in the spirit of \eqref{eq:BBM-limit}. 

Our main result confirms this intuition by identifying the pointwise BBM limit of the operator appearing on the left-hand side of \eqref{eq:HMP}.

\begin{theorem}\label{thm:1}
Let $f\in C_c^\infty(\R^n)$ with $n\ge 2$. Then for every $x\in\R^n$,
\begin{equation}\label{eq:main-limit}
\lim_{\alpha\to 1^-} (1-\alpha)\, I_\alpha(\mathcal{D}^\alpha f)(x) =K_n\, I_1(|\nabla f|)(x),
\end{equation}
where $K_n$ is given by \eqref{eq:Kn}.
\end{theorem}

\begin{remark}\label{rem:Kn}
The constant $K_n$ coincides with the geometric constant appearing in the BBM limit \eqref{eq:BBM-limit}, and it is independent of the choice of $e\in \mathbb{S}^{n-1}$ by rotation invariance. 
\end{remark}

In particular, Theorem \ref{thm:1} shows that the inequality \eqref{eq:HMP}, proved in \cite{HMP}, is asymptotically sharp as $\alpha\to 1^-$.

\subsection*{Outline of the paper}

The following section is devoted to several auxiliary lemmas. Section 3 is devoted to the proof of Theorem \ref{thm:1}. In Section 4, we extend the pointwise convergence in \eqref{eq:main-limit} from $C_c^\infty(\R^n)$ to $W^{1,1}(\R^n)$ by a density argument. Finally, we discuss further extensions, including an $L^p$-variant of the operator and its corresponding pointwise BBM limit.

\section{Auxiliary lemmas}
As a first step toward the proof of the main result, we analyze the pointwise limit of the fractional gradient $\mathcal{D}^\alpha f$. The following lemma describes this behavior for smooth functions.

\begin{lemma}\label{lem:1}
	Let $f\in C_c^2(\R^n)$, then for every $x\in \R^n$ we have
	\begin{equation*}
		\lim_{\alpha\to 1^-} (1-\alpha) \mathcal{D} ^\alpha f(x) = K_n\, |\nabla f(x)|. 	
	\end{equation*}
\end{lemma}

The previous result is implicit in \cite{BBM1}; we include a proof for the convenience of the reader. We refer to \cite{Kaushik} for a more general version involving general domains and mixed Sobolev seminorms $W_q^{\alpha, p}$.

\begin{proof}
	Fix $x\in \R^n$. We split the fractional derivative into two terms, the local and nonlocal parts,
	\begin{align*}
		(1-\alpha) \mathcal{D} ^\alpha f(x) = & (1-\alpha) \int_{\R^n} \frac{|f(x+h)-f(x)|}{|h|^{n+\alpha}}dh\\
		= & (1-\alpha) \int_{|h|<1} \frac{|f(x+h)-f(x)|}{|h|^{n+\alpha}}dh\\
		& +(1-\alpha) \int_{|h|\ge 1} \frac{|f(x+h)-f(x)|}{|h|^{n+\alpha}}dh\\
		= & I_1 + I_2.
	\end{align*}
	
	First, we observe that the nonlocal part does not contribute to the limit, 
	\begin{align*}
		I_2 \le & 2\|f\|_{L^\infty} (1-\alpha)\int_{|h|\ge 1} \frac{1}{|h|^{n+\alpha}}dh \\
		= & C_n \|f\|_{L^\infty} (1-\alpha)\int_{1}^\infty \frac{1}{r^{1+\alpha}}dr \\
		= & C_n \|f\|_{L^\infty}\frac{1-\alpha}{\alpha} \longrightarrow 0, 
	\end{align*}
	when $\alpha\to 1^-$. To study the local part, we consider the Taylor expansion of $f$ of order $2$,
	\begin{equation*}
		f(x+h)= f(x)+ \nabla f(x) \cdot h + R(h),
	\end{equation*}
	where the remainder term satisfies $|R(h)|\le \tfrac{1}{2}\|D^2 f\|_{L^\infty}\,|h|^2$ for $|h|<1$. Using the elementary inequality $\bigl||a+b|-|a|\bigr|\le |b|$ for $a,b\in\R$, with $a=\nabla f(x)\cdot h$ and $b=R(h)$, we obtain
	\begin{equation*}
		\bigl||f(x+h)-f(x)|-|\nabla f(x)\cdot h|\bigr|\le |R(h)|,\qquad |h|<1.
	\end{equation*}
	Therefore,
	\begin{equation}\label{eq:sandwich}
		|I_1-I_{11}|\le I_{12},
	\end{equation}
	where
	\begin{align*}
		I_{11}:=(1-\alpha)\int_{|h|<1}\frac{|\nabla f(x)\cdot h|}{|h|^{n+\alpha}}\,dh,
		\quad \text{ and } \quad
		I_{12}:=(1-\alpha)\int_{|h|<1}\frac{|R(h)|}{|h|^{n+\alpha}}\,dh.
	\end{align*}
	
	The remainder term satisfies
	\begin{align*}
		I_{12} &\le C_f(1-\alpha)\int_{|h|<1}\frac{|h|^2}{|h|^{n+\alpha}}\,dh\\
		&= C_f C_n (1-\alpha)\int_0^1 r^{1-\alpha}\,dr \\
		&= C_f C_n \frac{1-\alpha}{2-\alpha}\longrightarrow 0,
	\end{align*}
	when $\alpha\to 1^-$. Next, we compute $I_{11}$. If $\nabla f(x)=0$, then $I_{11}=0$ and \eqref{eq:sandwich} gives $0\le I_1\le I_{12}\longrightarrow 0$, and then $\lim_{\alpha\to 1^-}(1-\alpha)\mathcal{D}^\alpha f(x)=0=K_n|\nabla f(x)|$. Assume now that $\nabla f(x)\neq 0$. Using polar coordinates $h=r\omega$, we get
	\begin{align*}
		I_{11}
		&=(1-\alpha)\int_0^1\int_{\mathbb{S}^{n-1}}
		\frac{|\nabla f(x)\cdot (r\omega)|}{r^{n+\alpha}}\,r^{n-1}\,d\sigma(\omega)\,dr\\
		&=(1-\alpha)\left(\int_{\mathbb{S}^{n-1}}|\nabla f(x)\cdot \omega|\,d\sigma(\omega)\right)\int_0^1 r^{-\alpha}\,dr\\
		&=(1-\alpha)\left(\int_{\mathbb{S}^{n-1}}|\nabla f(x)\cdot \omega|\,d\sigma(\omega)\right)\frac{1}{1-\alpha}\\
		&=\int_{\mathbb{S}^{n-1}}|\nabla f(x)\cdot \omega|\,d\sigma(\omega).
	\end{align*}
	By homogeneity and rotational invariance,
	\begin{equation*}
		\int_{\mathbb{S}^{n-1}}|\nabla f(x)\cdot \omega|\,d\sigma(\omega)
		=|\nabla f(x)|\int_{\mathbb{S}^{n-1}}\left|\frac{\nabla f(x)}{|\nabla f(x)|}\cdot \omega\right|\,d\sigma(\omega)
		=K_n|\nabla f(x)|.
	\end{equation*}
	Therefore $I_{11}=K_n|\nabla f(x)|$ and, since $I_{12}\to 0$, \eqref{eq:sandwich} yields $\lim_{\alpha\to 1^-} I_1 = K_n|\nabla f(x)|$. Together with $I_2\to 0$ we conclude that
	\begin{equation*}
		\lim_{\alpha\to 1^-}(1-\alpha)\mathcal{D}^\alpha f(x)
		=\lim_{\alpha\to 1^-}(I_1+I_2)=K_n|\nabla f(x)|.
	\end{equation*}
	This concludes the proof. 
\end{proof}

\begin{remark}
As demonstrated in the proof, the nonlocal term vanishes in the limit. Actually, we have established the following result:
	\begin{equation*}
		\lim_{\alpha\to 1^-} (1-\alpha) \int_{B(x,r)} \frac{|f(x)-f(y)|}{|x-y|^{n+\alpha}}dy = K_n\, |\nabla f(x)|,
	\end{equation*}
	for every $x\in \R^n$ and every $0<r\le \infty$. In particular, if $\Omega$ is a domain and $\tau \in (0,1)$ we have proved 
	\begin{equation}\label{eq:domains}
		\lim_{\alpha\to 1^-} (1-\alpha) \int_{B(x,\tau \dist (x, \partial \Omega))} \frac{|f(x)-f(y)|}{|x-y|^{n+\alpha}}dy = K_n\, |\nabla f(x)|,
	\end{equation}
	for every $x\in \Omega$. 
\end{remark}

To justify the exchange of limits and integrals later in the proof of Theorem \ref{thm:1}, we need a uniform pointwise bound for $(1-\alpha) \mathcal{D}^\alpha f$ that holds uniformly for $\alpha$ close to $1$.

\begin{lemma}\label{lem:2}
	Let $f\in C^2_c(\R^n)$.  Then for each $\alpha \in (\tfrac{1}{2}, 1)$ and $x\in \R^n,$ we have
	\begin{equation*}
		(1-\alpha) \mathcal{D} ^\alpha f(x) \le C_n \pa{\|f\|_{L^\infty} +  \|\nabla f\|_{L^\infty}}.
	\end{equation*}
\end{lemma}

\begin{proof}
	We again split the fractional derivative into two terms: the local and nonlocal parts. We have
	\begin{align*}
		(1-\alpha) \mathcal{D} ^\alpha f(x) = &  (1-\alpha) \int_{|x-y|<1} \frac{|f(x)-f(y)|}{|x-y|^{n+\alpha}}dy \\
		&+ (1-\alpha) \int_{|x-y|\ge 1} \frac{|f(x)-f(y)|}{|x-y|^{n+\alpha}}dy\\
		= & I_1 + I_2.
	\end{align*}
	
	To estimate the local term, by the mean value theorem, we have
	\begin{align*}
		I_1 \le & C \| \nabla f\|_{L^\infty} (1-\alpha) \int_{|x-y|<1} \frac{|x-y|}{|x-y|^{n+\alpha}}dy \\
		= & C_n \| \nabla f\|_{L^\infty} (1-\alpha) \int_0^1 r^{-\alpha} dr\\
		= & C_n \| \nabla f\|_{L^\infty}.
	\end{align*}
	
	On the other hand, we use the fact that $f$ is bounded,
	\begin{align*}
		I_2 \le & 2\|f\|_{L^\infty} (1-\alpha) \int_{|x-y|\ge 1} \frac{1}{|x-y|^{n+\alpha}}dy\\
		 = & C_n \|f\|_{L^\infty} (1-\alpha) \int_1^\infty r^{-1-\alpha} dr \\
		 = & C_n \|f\|_{L^\infty} \frac{1-\alpha}{\alpha} \\
		 \le & C_n \|f\|_{L^\infty}
	\end{align*}
	where in the last inequality we use $\alpha \in (\tfrac{1}{2}, 1)$. 
\end{proof}

\begin{remark}
Let $\Omega\subset\R^n$ be an arbitrary bounded domain and fix $\tau\in(0,1)$. Combining \eqref{eq:domains} with the uniform bound from Lemma \ref{lem:2}, one may apply the dominated convergence theorem to obtain the well-known BBM limit in $\Omega$ with the truncated seminorm
\begin{equation*}
	\lim_{\alpha \to 1^-} (1-\alpha)\int_{\Omega}\int_{B(x,\tau\,\dist(x,\partial\Omega))}\frac{|f(x)-f(y)|}{|x-y|^{n+\alpha}}\,dy\,d\mu(x) =K_n\int_{\Omega}|\nabla f(x)|\,d\mu(x),
\end{equation*}
where $\mu$ is a locally finite measure. In particular, this recovers the main result of \cite{DrelichmanDuran} for $p=1$ and $f\in C_c^2(\R^n)$.
\end{remark}

\section{Proof of Theorem \texorpdfstring{\ref{thm:1}}{1.1}}

\begin{proof}[Proof of Theorem \ref{thm:1}]
	Fix $x\in\R^n$ and set $f_x(u):=f(u+x)$. Since both $\mathcal{D}^\alpha$ and the Riesz potential $I_\alpha$ are translation invariant, we have, for every $\alpha\in(0,1)$,
	\begin{equation*}
		(1-\alpha)\,I_\alpha(\mathcal{D}^\alpha f)(x)=(1-\alpha)\,I_\alpha(\mathcal{D}^\alpha f_x)(0).
	\end{equation*}
	Moreover, $|\nabla f_x(u)|=|\nabla f(u+x)|$, and therefore
	\begin{equation*}
		I_1(|\nabla f|)(x)=I_1(|\nabla f_x|)(0).
	\end{equation*}
	Hence, it is enough to prove \eqref{eq:main-limit} in the case $x=0$, applied to the translated function $f_x$. In what follows, we assume $x=0$ and write $f$ in place of $f_x$ for simplicity.
	
	For each $\alpha \in (0,1)$, define
	\begin{equation*}
		F_\alpha(y):=\gamma_{n,\alpha}(1-\alpha)\frac{\mathcal{D}^\alpha f(y)}{|y|^{n-\alpha}},
	\end{equation*}
	and
	\begin{equation*}
		F(y):=K_n\gamma_{n,1}\frac{|\nabla f(y)|}{|y|^{n-1}},
	\end{equation*}
	for $y\in\R^n\setminus\{0\}$, where $\gamma_{n,\alpha}$ is defined in \eqref{Defgamma}. By Lemma \ref{lem:1}, together with the continuity of $\gamma_{n,\alpha}$ as a function of $\alpha$, we have
	\begin{equation*}
		\lim_{\alpha\to1^-}F_\alpha(y)=F(y)
	\end{equation*}
	for every $y\in\R^n\setminus\{0\}$. Thus, it remains to show that
	\begin{equation}\label{eq:limint}
		\lim_{\alpha\to1^-}\int_{\R^n}F_\alpha(y)\,dy=\int_{\R^n}F(y)\,dy.
	\end{equation}
	To justify the interchange of limit and integral, we shall apply the dominated convergence theorem. Let $R_0>0$ be such that $\supp(f)\subset B(0,R_0)$, and decompose
	\begin{equation*}
		\R^n=A_1\cup A_2\cup A_3,
	\end{equation*}
	where
	\begin{equation*}
		A_1=B(0,1), \qquad A_2=\R^n\setminus B(0,2R_0+1), \qquad A_3=B(0,2R_0+1)\setminus B(0,1).
	\end{equation*}
	
	We begin with $y\in A_1$. If $\alpha\in(\tfrac12,1)$, then Lemma \ref{lem:2} gives
	\begin{align*}
		F_\alpha(y)
		&=\gamma_{n,\alpha}(1-\alpha)\frac{\mathcal{D}^\alpha f(y)}{|y|^{n-\alpha}}\\
		&\le C_n\frac{\|f\|_{L^\infty}+\|\nabla f\|_{L^\infty}}{|y|^{n-\alpha}}\\
		&\le C_{n,f}\frac{1}{|y|^{n-\frac12}}\\
		& =: h(y).
	\end{align*}
	Moreover, $h\in L^1(A_1)$, since
	\begin{equation*}
		\int_{A_1}h(y)\,dy=\int_{B(0,1)}\frac{1}{|y|^{n-\frac12}}\,dy
		=C_n\int_0^1r^{-\frac12}\,dr<\infty.
	\end{equation*}
	
	Next, let $y\in A_2$. Then $f(y)=0$, and therefore
	\begin{align*}
		\mathcal{D}^\alpha f(y)
		&=\int_{\R^n}\frac{|f(y)-f(z)|}{|y-z|^{n+\alpha}}\,dz\\
		&=\int_{\supp(f)}\frac{|f(z)|}{|y-z|^{n+\alpha}}\,dz\\
		&\le \int_{\supp(f)}|f(z)|\frac{2^{n+\alpha}}{|y|^{n+\alpha}}\,dz\\
		&=2^{n+\alpha}\|f\|_{L^1}\frac{1}{|y|^{n+\alpha}},
	\end{align*}
	because $z\in \supp(f)\subset B(0,R_0)$ and $|y|\ge 2R_0+1$, so 	
	\begin{equation*}
		|y-z|\ge |y|-|z|\ge |y|-R_0\ge \frac{|y|}{2}.
	\end{equation*}
	Hence, for $y\in A_2$ 
	\begin{align*}
		F_\alpha(y)
		&=\gamma_{n,\alpha}(1-\alpha)\frac{\mathcal{D}^\alpha f(y)}{|y|^{n-\alpha}}\\
		&\le \gamma_{n,\alpha}(1-\alpha)2^{n+\alpha}\|f\|_{L^1}\frac{1}{|y|^{n+\alpha}|y|^{n-\alpha}}\\
		&\le C_n\|f\|_{L^1}\frac{1}{|y|^{2n}}\\
		& =:g(y),
	\end{align*}
	and $g\in L^1(A_2)$, since
	\begin{equation*}
		\int_{A_2}g(y)\,dy
		=C_{n,f}\int_{\R^n\setminus B(0,2R_0+1)}\frac{1}{|y|^{2n}}\,dy
		=C_{n,f}\int_{2R_0+1}^\infty r^{-n-1}\,dr<\infty.
	\end{equation*}
	
	Finally, let $y\in A_3$. Since $|y|\ge 1$, Lemma \ref{lem:2} yields
	\begin{equation*}
		F_\alpha(y)=\gamma_{n,\alpha}(1-\alpha)\frac{\mathcal{D}^\alpha f(y)}{|y|^{n-\alpha}}\le C_{n,f},
	\end{equation*}
	for every $\alpha\in(\tfrac12,1)$. Since $|A_3|<\infty$, this gives an integrable dominating function on $A_3$.
	
	Combining the estimates on $A_1$, $A_2$, and $A_3$, we obtain an $L^1$-dominating function on $\R^n$. Therefore, the dominated convergence theorem yields \eqref{eq:limint} which completes the proof.
\end{proof}

\section{Extension to the Sobolev space \texorpdfstring{$W^{1,1}(\R^n)$}{W11}}

In this section, we extend Theorem \ref{thm:1} to the Sobolev space $W^{1,1}(\R^n)$ by a density argument. First, we prove the following inequality for $p=1$ and the global Gagliardo seminorm.

\begin{proposition}\label{prop:inverse}
	Let $f\in W^{1,1}(\R^n)$ and let $0<\alpha<1$. Then
	\begin{equation*}
		\alpha(1-\alpha)\int_{\R^n}\int_{\R^n} \frac{|f(x)-f(y)|}{|x-y|^{n+\alpha}}\,dy\,dx \le C_n \|f\|_{W^{1,1}(\R^n)}.
	\end{equation*}
\end{proposition}

This proposition is a consequence of the Gagliardo-Nirenberg interpolation inequality proved in \cite{BrezisMironescu} (see also \cite[(1.10)]{HLYY}). For completeness, we include a proof. 

\begin{proof}
	By the change of variables $y=x+h$,
	\begin{align*}
		\int_{\R^n}\int_{\R^n}\frac{|f(x)-f(y)|}{|x-y|^{n+\alpha}}\,dy\,dx = \int_{\R^n}\int_{\R^n}\frac{|f(x+h)-f(x)|}{|h|^{n+\alpha}}\,dx\,dh =: I_1+I_2,
	\end{align*}
	where in $I_1$ we integrate over $|h|<1$, while in $I_2$ we integrate over $|h|\ge 1$. For $I_1$, we use the following difference quotients estimate,
	\begin{equation*}
		\int_{\R^n}|f(x+h)-f(x)|\,dx \le |h|\int_{\R^n}|\nabla f(x)|\,dx,
	\end{equation*}
	for every $h\in\R^n$ (see for instance \cite[Proposition 9.3]{BrezisBook}). Indeed, this is immediate for $f\in C_c^\infty(\R^n)$ from
	\begin{equation*}
		f(x+h)-f(x)=\int_0^1 \nabla f(x+th)\cdot h\,dt,
	\end{equation*}
	and the general case follows by density in $W^{1,1}(\R^n)$. Hence
	\begin{align*}
		I_1 &\le \int_{|h|<1}\frac{1}{|h|^{n+\alpha-1}} \left(\int_{\R^n}|\nabla f(x)|\,dx\right)dh \le \frac{C_n}{1-\alpha}\int_{\R^n}|\nabla f(x)|\,dx.
	\end{align*}
	For $I_2$, the triangle inequality gives
	\begin{align*}
		I_2 &\le 2\int_{|h|\ge 1}\frac{dh}{|h|^{n+\alpha}}\int_{\R^n}|f(x)|\,dx \le \frac{C_n}{\alpha}\int_{\R^n}|f(x)|\,dx.
	\end{align*}
	Combining both estimates, we obtain
	\begin{equation*}
		\int_{\R^n}\int_{\R^n}\frac{|f(x)-f(y)|}{|x-y|^{n+\alpha}}\,dy\,dx \le \frac{C_n}{1-\alpha}\int_{\R^n}|\nabla f(x)|\,dx + \frac{C_n}{\alpha}\int_{\R^n}|f(x)|\,dx.
	\end{equation*}
	We conclude the proof by multiplying both sides by $\alpha(1-\alpha)$ and using that $\alpha < 1$ and $1-\alpha < 1$.
\end{proof}

We can now state the main result of this section.

\begin{theorem}
	Let $f\in W^{1,1}(\R^n)$ with $n\ge 2$. Then, for every sequence $\{\alpha_k\}_{k} \subset (0,1)$ with $\alpha_k \to 1^-$, there exists a subsequence $\{\alpha_{k_j}\}_{j}$ such that
	\begin{equation*}
		\lim_{j\to \infty} (1-\alpha_{k_j}) I_{\alpha_{k_j}} (\mathcal{D} ^{\alpha_{k_j}} f)(x) = K_n\, I_1(|\nabla f|)(x),
	\end{equation*}
	for almost every $x\in \R^n$.
\end{theorem}

\begin{proof}
	Let $f\in W^{1,1}(\R^n)$. By the density of smooth compactly supported functions, there exists a sequence $\{\varphi_k\}_{k=1}^\infty\subset C_c^\infty(\R^n)$ such that $\varphi_k\to f$ in $W^{1,1}(\R^n)$ (see for instance \cite{AdamsFournier, BrezisBook}). We first show that
	\begin{equation*}
		(1-\alpha) I_\alpha(\mathcal{D}^\alpha f)\to K_n I_1(|\nabla f|)
	\end{equation*}
	in measure on every compact set $E \subset \R^n$ as $\alpha\to 1^-$. 
	
	For every $k\in \mathbb{N}$ and every $\alpha\in(0,1)$, we use the triangle inequality to write
	\begin{align*}
		\left|(1-\alpha) I_\alpha(\mathcal{D}^\alpha f)(x)-K_n I_1(|\nabla f|)(x)\right|
		\le &\ (1-\alpha)\left|I_\alpha(\mathcal{D}^\alpha f)(x)-I_\alpha(\mathcal{D}^\alpha \varphi_k)(x)\right|\\
		&+\left|(1-\alpha)I_\alpha(\mathcal{D}^\alpha \varphi_k)(x)-K_n I_1(|\nabla \varphi_k|)(x)\right|\\
		&+K_n\left|I_1(|\nabla \varphi_k|)(x)-I_1(|\nabla f|)(x)\right|\\
		=:&\ A_{1,\alpha,k}(x)+A_{2,\alpha,k}(x)+A_{3,k}(x).
	\end{align*}
	We begin by bounding the first term. Combining the weak-type estimate for the Riesz potential with Proposition \ref{prop:inverse} for $\alpha \in (\tfrac12, 1)$, we obtain 
	\begin{align*}
		\|A_{1,\alpha,k}\|_{L^{\frac{n}{n-\alpha},\infty}(\R^n)}
		&\le (1-\alpha)\|I_\alpha(\mathcal{D}^\alpha(f-\varphi_k))\|_{L^{\frac{n}{n-\alpha},\infty}(\R^n)}\\
		&\le C_n (1-\alpha)\int_{\R^n}\mathcal{D}^\alpha(f-\varphi_k)(x)\,dx\\
		&\le C_n \|f-\varphi_k\|_{W^{1,1}(\R^n)}.
	\end{align*}
	Similarly, for the third term,
	\begin{align*}
		\|A_{3,k}\|_{L^{\frac{n}{n-1},\infty}(\R^n)}	&\le  K_n \|I_1(|\nabla \varphi_k-\nabla f|)\|_{L^{\frac{n}{n-1},\infty}(\R^n)}\\
		&\le C_n \|\nabla \varphi_k-\nabla f\|_{L^1(\R^n)}\\
		&= C_n \|\nabla (\varphi_k- f)\|_{L^1(\R^n)}\\
		&\le C_n \|f-\varphi_k\|_{W^{1,1}(\R^n)}.
	\end{align*}
	Given any $\varepsilon>0$ and $\eta>0$, since $\varphi_k \to f$ in $W^{1,1}(\R^n)$, we can choose $k$ sufficiently large such that
	\begin{equation*}
		\frac{C_n\|f-\varphi_k\|_{W^{1,1}(\R^n)}}{\varepsilon} < \min\left\{ 1,\, \frac{\eta}{3} ,\, 
		\left(\frac{\eta}{3}\right)^{\frac{n-1}{n}}\right\}. 
	\end{equation*}
	Fix a compact set $E\subset \R^n$. Since $\frac{n}{n-\alpha}\ge 1$ for every $\alpha\in(0,1)$, applying Chebyshev's inequality to $A_{1,\alpha,k}$ we have
	\begin{align*}
		\left|\left\{x\in E:\ A_{1,\alpha,k}(x)>\varepsilon\right\}\right|
		&\le \left|\left\{x\in \R^n:\ A_{1,\alpha,k}(x)>\varepsilon\right\}\right|\\
		&\le \left(\frac{\|A_{1,\alpha,k}\|_{L^{\frac{n}{n-\alpha},\infty}(\R^n)}}{\varepsilon}\right)^{\frac{n}{n-\alpha}} \\
		&\le \left(\frac{C_n\|f-\varphi_k\|_{W^{1,1}(\R^n)}}{\varepsilon}\right)^{\frac{n}{n-\alpha}}\\
		& < \frac{\eta}{3}.
	\end{align*}
	Applying Chebyshev's inequality to $A_{3,k}$, we also obtain
	\begin{align*}
			\left|\left\{x\in E:\ A_{3,k}(x)>\varepsilon\right\}\right|  \le & \left(\frac{\|A_{3,k}\|_{L^{\frac{n}{n-1},\infty}(\R^n)}}{\varepsilon}\right)^{\frac{n}{n-1}} \\
			\le & \left(\frac{C_n \|f-\varphi_k\|_{W^{1,1}(\R^n)}}{\varepsilon}\right)^{\frac{n}{n-1}} \\ 
			<& \frac{\eta}{3}.
	\end{align*}
	For this fixed choice of $k$, since $\varphi_k \in C_c^\infty(\mathbb{R}^n)$, we may apply Theorem~\ref{thm:1} to conclude that $A_{2,\alpha,k}(x) \to 0$ for every $x \in \mathbb{R}^n$ as $\alpha \to 1^-$. Since $|E|<\infty$, pointwise convergence implies convergence in measure on $E$. Thus, there exists $\alpha_0\in(\tfrac12,1)$ (depending on $\varepsilon$ and $\eta$) such that for every $\alpha\in(\alpha_0,1)$,
	\begin{equation*}
		\left|\left\{x\in E:\ A_{2,\alpha,k}(x)>\varepsilon\right\}\right|<\frac{\eta}{3}.
	\end{equation*}
	Combining the estimates for $A_{1,\alpha,k}$, $A_{2,\alpha,k}$, and $A_{3,k}$ yields
	\begin{equation*}
		\left|\left\{x\in E:\ \left|(1-\alpha) I_\alpha(\mathcal{D}^\alpha f)(x)-K_n I_1(|\nabla f|)(x)\right|>3\varepsilon\right\}\right|<\eta
	\end{equation*}
	for every $\alpha\in(\alpha_0,1)$. Since $\varepsilon, \eta>0$ were arbitrary, this proves that $(1-\alpha) I_\alpha(\mathcal{D}^\alpha f)\to K_n I_1(|\nabla f|)$ in measure on every compact set $E\subset \R^n$ as $\alpha\to1^-$.
	
	Finally, let $\{\alpha_k\}_{k}\subset(0,1)$ be any sequence such that $\alpha_k\to1^-$. By the previous step, the sequence $\{(1-\alpha_k) I_{\alpha_k}(\mathcal{D}^{\alpha_k}f)\}_{k}$ converges to $K_n I_1(|\nabla f|)$ in measure on every compact subset of $\R^n$. Therefore, there exists a subsequence $\{\alpha_{k_j}\}_{j}$ such that
	\begin{equation*}
		(1-\alpha_{k_j}) I_{\alpha_{k_j}}(\mathcal{D}^{\alpha_{k_j}}f)(x) \to K_n I_1(|\nabla f|)(x)
	\end{equation*}
	for almost every $x\in\R^n$. This completes the proof.
\end{proof}

\section{Further extensions}

In this final section, we record a natural $L^p$-variant of the nonlinear fractional differential operator. For $0<\alpha<1$ and $p\ge1$, we define
\begin{equation*}
\mathcal{D}^\alpha_p f(x):=\left(\int_{\R^n}\frac{|f(x)-f(y)|^p}{|x-y|^{n+\alpha p}}\,dy\right)^{\frac{1}{p}}.
\end{equation*}
This operator (in particular for $p=2$) appears, for instance, in \cite{NahasPonce}. This operator also behaves as a differential operator; it satisfies the following Leibniz-type inequality \begin{equation*}
	\norm{\mathcal{D}^\alpha_p (fg)}_{L^p(\R^n)}\le \norm{f\, \mathcal{D}^\alpha_p g}_{L^p(\R^n)}+\norm{g\, \mathcal{D}^\alpha_p f}_{L^p(\R^n)}.
\end{equation*}

We next state the corresponding pointwise BBM limit for the Riesz potential of $\mathcal{D}^\alpha_p$.

\begin{theorem}
Let $p\ge1$ and let $f\in C_c^\infty(\R^n)$. Then for every $x\in\R^n$,
\begin{equation*}
\lim_{\alpha\to 1^-} (1-\alpha)^{\frac{1}{p}}\, I_\alpha(\mathcal{D}^\alpha_p f)(x) =K_{n,p}\, I_1(|\nabla f|)(x),
\end{equation*}
where
\begin{equation*}
	K_{n,p}:=\left( \frac{1}{p} \int_{\mathbb{S}^{n-1}}|\omega\cdot e|^p\,d\sigma(\omega)   \right)^\frac{1}{p} .
\end{equation*}
\end{theorem}

\begin{proof}[Sketch of the proof]
	The proof follows the strategy of Theorem \ref{thm:1}. By translation invariance, it suffices to consider $x=0$. Setting
\begin{equation*}
	F_{\alpha,p}(y):=\gamma_{n,\alpha}\,(1-\alpha)^{\frac1p}\,\frac{\mathcal{D}^\alpha_p f(y)}{|y|^{n-\alpha}},
\qquad
F(y):=K_{n,p}\,\gamma_{n,1}\,\frac{|\nabla f(y)|}{|y|^{n-1}},
\end{equation*}
it is enough to justify that $\lim_{\alpha\to 1^-}\int_{\R^n}F_{\alpha,p}=\int_{\R^n}F$. The pointwise convergence $F_{\alpha,p}(y)\to F(y)$ is provided by \cite[Lemma 13]{Kaushik} (in place of Lemma~\ref{lem:1}) together with the continuity of $\gamma_{n,\alpha}$ in $\alpha$. The required domination to apply the dominated convergence theorem is obtained by the same splitting estimates as in the proof of Theorem \ref{thm:1}, adapted to $\mathcal{D}^\alpha_p f$.
\end{proof}

%\section*{Acknowledgment}

%\section*{Conflicts of Interest}
%The author has no conflicts of interest to declare.

%\section*{Data Availability Statement}
%Data sharing is not applicable to this article, as no datasets were generated or analyzed during the current study.

\end{document}